\documentclass[11pt,twoside]{article}
\usepackage[left=25.4mm, right=25.4mm, top=25.4mm, bottom=25.4mm]{geometry}   
\usepackage[utf8]{inputenc}                              
\usepackage[english]{babel}                              
\usepackage[hidelinks]{hyperref}                         
\usepackage{booktabs, multirow}                          
\usepackage{amsmath, amsfonts, stmaryrd, amssymb, amsbsy}        
\usepackage{caption, subcaption}   
\usepackage{graphicx}
\usepackage[all]{xy}
\usepackage{tikz-cd}                                     
\usepackage{enumitem}                                    
\usepackage{xcolor, graphicx}                            
\usepackage[noblocks]{authblk}                           
\usepackage{titlesec}                                    
\usepackage{fancyhdr}                                    
\usepackage{etoolbox,textcomp}
\usepackage[noadjust]{cite}
\apptocmd{\thebibliography}{\setlength{\itemsep}{-3pt}}{}{}
\usepackage{amsthm}
\theoremstyle{plain}
\newtheorem{theorem}{Theorem}[section]
\newtheorem{lemma}[theorem]{Lemma}
\newtheorem{proposition}[theorem]{Proposition}
\newtheorem{corollary}[theorem]{Corollary}

\theoremstyle{definition}

\theoremstyle{remark}

\newcommand{\R}{\mathbb R}
\newcommand{\Z}{\mathbb Z}

\newcommand{\Sph}{\mathbb S}
\newcommand{\Prj}{\mathbb P}
\newcommand{\G}{\mathcal G}
\newcommand{\J}{\mathcal J}

\newcommand{\F}{\mathcal F}

\newcommand{\xm}{\mathfrak X(M)}

\newcommand{\dm}{\Omega^1 (M)}

\newcommand{\GTM}{\Gamma (\mathbb TM)}
\newcommand{\TM}{\mathbb TM}

\makeatletter
\def\blfootnote{\gdef\@thefnmark{}\@footnotetext}
\makeatother

\title{\bf On the triviality of the generalized tangent bundle
  \blfootnote{\em E-mail addresses:}
  \blfootnote{Fernando Etayo: \href{mailto:fernando.etayo@unican.es}{fernando.etayo@unican.es}}
  \blfootnote{Pablo Gómez-Nicolás: \href{mailto:pablogomeznicolas@gmail.com}{pablogomeznicolas@gmail.com}}
  \blfootnote{Rafael Santamaría: \href{mailto:rsans@unileon.es}{rsans@unileon.es}}
}
\author[*]{Fernando Etayo}
\author[ \hspace{-1ex}]{Pablo Gómez-Nicolás}
\author[$\dag$]{Rafael Santamaría}
\affil[*]{\footnotesize Departamento de Matemáticas, Estadística y Computación, Facultad de Ciencias, Universidad de Cantabria, Avda. de los Castros, s/n, 39071, Santander, Spain}
\affil[$\dag$]{\footnotesize Departamento de Matemáticas, Escuela de Ingenierías Industrial, Informática y Aeroespacial, Universidad de León, Campus de Vegazana, 24071, León, Spain}
\date{\small \today}

\begin{document}

\maketitle

\begin{abstract}
\noindent We study the relations between the triviality of the tangent bundle $TM$ and the generalized tangent bundle $\TM = TM\oplus T^*M$ of a manifold. We show that the generalized tangent bundle of a paralellizable manifold is trivial. We also prove that the converse implication does not hold, by studying the cases of the Möbius strip, spheres and projective spaces. Finally, we relate the triviality of the generalized tangent bundle to generalized geometric structures.
\end{abstract}



\thispagestyle{empty}
\section{Introduction}
\label{SECTION:INTRODUCTION}

The \emph{generalized tangent bundle} of a smooth manifold $M$, introduced by T. J. Courant in \cite{COURANT1990} and later studied in depth by N. Hitchin, G. Cavalcanti and M. Gualtieri in \cite{CAVALCANTI2004, GUALTIERI2004, HITCHIN2003}, is defined as the Whitney sum of its tangent and cotangent bundles, namely, $\TM := TM\oplus T^*M\to M$. Sections of $\TM$ can be seen as sums of vector fields and 1-forms, that is, $\GTM = \xm\oplus \dm$. Much of the existing research in generalized geometry focuses on geometric structures on $\TM$, such as generalized almost complex and generalized almost paracomplex structures (see \cite{GOMEZNICOLAS2025,IDAMANEA2017,NANNICINI2010,WADE2004}). However, there appears to be relatively little research focusing on specific manifolds; some studies that do address particular manifolds include \cite{CAVALCANTIGUALTIERI2006}, concerning 4-dimensional manifolds, and \cite{ETAYOGOMEZNICOLASSANTAMARIA2025}, studying a version of the Hopf problem on $\Sph^6$ in generalized geometry. Moreover, as far as we know, there are no studies relating the topological properties of $\TM$ to those of the underlying manifold.

A noticeable topological property of a smooth manifold is its parallelizability. A vector bundle $E\to M$ of rank $k$ over a manifold $M$ is said to be \emph{trivial} if it is isomorphic as a vector bundle to $M\times \R^k\to M$; equivalently, if there exist $k$ global sections $X^1,\ldots, X^k\in \Gamma(E)$ that are pointwise linearly independent. If the tangent bundle $TM\to M$ of a manifold is trivial, then $M$ is called \emph{parallelizable}. Examples of parallelizable manifolds include the spheres $\Sph^1$, $\Sph^3$ and $\Sph^7$, the projective spaces $\R\Prj^1$, $\R\Prj^3$ and $\R\Prj^7$, Lie groups, smooth closed orientable 3-dimensional manifolds, and contractible manifolds; in fact, every vector bundle over a contractible manifold is known to be trivial. Various invariants can be used to measure the non-triviality of a vector bundle. For example, it is well known that the vanishing of all the Stiefel-Whitney classes $w_i(E)\in H^i(M;\Z_2)$ for $i > 0$ is a necessary condition for the bundle to be trivial, but this condition is not sufficient.

In this brief note we study how the triviality of $TM$ is related to that of $\TM$. In Section \ref{SECTION:TRIVIALITYGENERALIZEDTANGENTBUNDLE} we prove that the triviality of the tangent bundle of a manifold implies the triviality of its generalized tangent bundle, and consequently the generalized tangent bundle of a parallelizable manifold is trivial. After that, we show that the converse implication does not hold by proving that the generalized tangent bundle of the Möbius strip $S$ is trivial. The idea of the proof is to construct four nowhere-vanishing sections of $\mathbb TS$ from vector fields that vanish at different points, by transforming these vector fields into 1-forms using an auxiliary Riemannian metric. We also show that the generalized tangent bundle of any sphere is trivial. We then use Stiefel-Whitney classes to study the case of the real projective spaces $\R\Prj^n$ such that $n+1$ is not a power of $2$, which are not parallelizable and whose generalized tangent bundle are non-trivial. Table \ref{tab:introduction} summarizes the main results and examples, comparing the triviality of the tangent bundle of a manifold with that of its generalized tangent bundle. Finally, in Section \ref{SECTION:RELATIONOFTRIVIALITYWITHGENERALIZEDGEOMETRICSTRUCTURES} we relate the triviality of the generalized tangent bundle to different generalized geometric structures previously studied in the literature.

\begin{table}[h]
\begin{center}
\begin{tabular}{lcc}
\toprule
    			 & {\bf$\TM$ trivial} & {\bf$\TM$ non-trivial} \\
\midrule
\multirow{5}{*}{{\bf$TM$ trivial}}	   & $\Sph^1$, $\Sph^3$, $\Sph^7$ 				& \multicolumn{1}{c}{\multirow{5}{*}{No}}  \\
 								  	   & $\R\Prj^1$, $\R\Prj^3$, $\R\Prj^7$  		&										   \\
 								  	   & Lie groups  								&										   \\
 								  	   & 3-dimensional closed orientable manifolds  &										   \\
 								  	   & Contractible manifolds			 			&										   \\
\midrule
\multirow{3}{*}{{\bf$TM$ non-trivial}} & Möbius strip	 	  	   & \multicolumn{1}{c}{\multirow{3}{*}{$\R\Prj^n$ for $n\neq 2^k - 1$}}	\\
 								  	   & Klein bottle 			   &														\\
 								  	   & $\Sph^n$ for $n\neq1,3,7$ &														\\
\bottomrule
\end{tabular}
\caption{Summary of the relations between the triviality of $TM$ and $\TM$.}
\label{tab:introduction}
\end{center}
\end{table}

\section{Triviality of the generalized tangent bundle}
\label{SECTION:TRIVIALITYGENERALIZEDTANGENTBUNDLE}

\subsection{Parallelizable manifolds}
\label{SUBSECTION:PARALLELIZABLEMANIFOLDS}

We begin by assuming that $M$ is a real parallelizable $n$-dimensional smooth manifold (Hausdorff and second countable).
\begin{proposition}
\label{prop:parallelizableimpliestriviality}
The generalized tangent bundle $\TM$ of a parallelizable manifold $M$ is trivial.
\end{proposition}

\begin{proof}
Since $M$ is parallelizable, there exist $n$ vector fields $X^1,\ldots,X^n\in \xm$ which are pointwise linearly independent. Then, we can introduce $n$ 1-forms using the isomorphism $\flat_g\colon \xm\to \dm$ induced by a Riemannian metric $g$ on $M$, such that $(\flat_g X)Y = g(X, Y)$ for every $X,Y\in \xm$. The differential forms $\flat_gX^1,\ldots,\flat_gX^n\in \dm$ are also pointwise linearly independent. Thus, the $2n$ sections $X^1,\flat_gX^1,X^2,\flat_gX^2,\ldots,X^n,\flat_gX^n\in \GTM$ are linearly independent at each point $p\in M$, which proves the result.
\end{proof}

This result implies that the generalized tangent bundle of manifolds such as Lie groups or the spheres $\Sph^1$, $\Sph^3$ and $\Sph^7$ is trivial. Other examples of parallelizable manifolds are listed in Table \ref{tab:introduction}.

\subsection{Non-parallelizable manifolds with trivial generalized tangent bundle}
\label{SUBSECTION:NONPARALLELIZABLEMANIFOLDSTRIVIALGENERALIZED}

The next results show that the triviality of the generalized tangent bundle of a manifold does not imply the triviality of its tangent bundle. We first study the Möbius strip, which can be defined as the quotient space
\begin{equation}
    S = \R\times (-1, 1) / {\sim}, \quad (u, v) \sim (u + m, (-1)^m v), \quad m\in\Z.
\label{eq:mobiusbanddefinition}
\end{equation}
This manifold is not orientable, hence $TS$ is not trivial. However, one can show that there exist four sections of $\mathbb TS$ that are pointwise linearly independent.

\begin{proposition}
\label{prop:trivialitymobius}
The generalized tangent bundle $\mathbb TS$ of the Möbius strip $S$ is trivial.
\end{proposition}

\begin{proof}
Using the definition of the Möbius strip given in Eq. \eqref{eq:mobiusbanddefinition}, one can endow $S$ with a smooth structure by considering the open subsets $U = S\setminus \{[(0, v)]\in S\}$ and $V = S\setminus \{[(1/2, v)]\in S\}$. The coordinate charts associated to $U$ and $V$, $x\colon U \longrightarrow \R^2$ and $y\colon V \longrightarrow \R^2$, are defined as follows:
\[
    x([(u, v)]) = (u, v) \enspace\text{with } 0<u<1, \quad\quad y([(u, v)]) = \left\lbrace
    \begin{array}{ll}
        (u, v)       & \text{if }0\leq u < 1/2,  \\
        (u - 1, -v)  & \text{if }1/2 < u < 1.
    \end{array}
    \right.
\]
Then, considering that $U\cap V = \{[(u, v)]\in S\colon 0<u<1/2\enspace \mathrm{or}\enspace 1/2<u<1\}$, the transition map $y\circ x^{-1}\colon x(U\cap V)\to y(U\cap V)$ is given by
\begin{equation}
    (y\circ x^{-1})(u, v) = \left\lbrace
    \begin{array}{ll}
        (u, v)       & \text{if }0 < u < 1/2,  \\
        (u - 1, -v)  & \text{if } 1/2 < u < 1.
    \end{array}
    \right.
\label{eq:transitionmapmobius}
\end{equation}
	
From Eq. \eqref{eq:transitionmapmobius}, it is immediate that the local coordinate fields associated to these charts fulfill the following relations for any $[(u,v)]\in U\cap V$:
\[
    \left.\frac{\partial}{\partial y^1}\right|_{[(u, v)]} = \left.\frac{\partial}{\partial x^1}\right|_{[(u, v)]}, \quad\quad \left.\frac{\partial}{\partial y^2}\right|_{[(u, v)]} = \left\lbrace
    \begin{array}{rl}
        \left.\dfrac{\partial}{\partial x^2}\right|_{[(u, v)]}   & \text{if }0 < u < 1/2,  \\
        -\left.\dfrac{\partial}{\partial x^2}\right|_{[(u, v)]}  & \text{if }1/2 < u < 1.
    \end{array}
    \right.
\]
We can define the following smooth vector fields on $\R^2$ using standard coordinates $(u, v)$:
\[
    X = \frac{\partial}{\partial u}, \quad Y = \cos(\pi u)\frac{\partial}{\partial v}, \quad Z = \sin(\pi u)\frac{\partial}{\partial v}.
\]
These vector fields are smooth on $\R^2$. If we restrict these vector fields to $\R\times (-1, 1)$ and use the quotient map $q\colon \R\times (-1, 1)\to S$, we must check that $q_*(X)$, $q_*(Y)$ and $q_*(Z)$ are invariant under the equivalence relation, so that they descend to well-defined vector fields on $S$.

\begin{figure}[hbt!]
\begin{center}
    \begin{subfigure}[h]{\textwidth}
        \centering
        \includegraphics[width=0.22\textwidth]{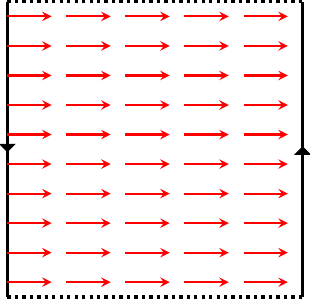}
        \hspace{1cm}
        \includegraphics[width=0.275\textwidth]{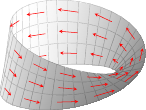}
        \caption{Vector field $q_*(X)$.}
        \label{fig:mobiusx}
    \end{subfigure}\par\vspace{.5cm}
    \begin{subfigure}[h]{\textwidth}
        \centering
        \includegraphics[width=0.22\textwidth]{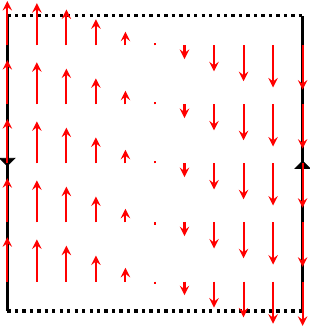}
        \hspace{1cm}
        \includegraphics[width=0.275\textwidth]{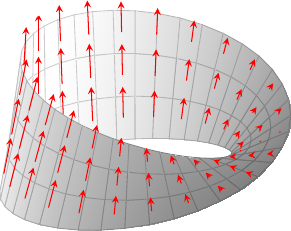}
        \caption{Vector field $q_*(Y)$.}
        \label{fig:mobiusy}
    \end{subfigure}\par\vspace{.5cm}
    \begin{subfigure}[h]{\textwidth}
        \centering
        \includegraphics[width=0.22\textwidth]{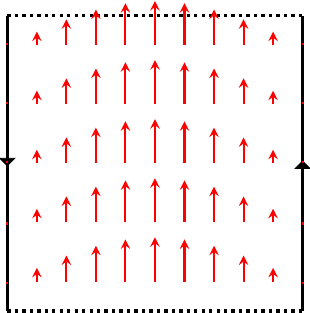}
        \hspace{1cm}
        \includegraphics[width=0.275\textwidth]{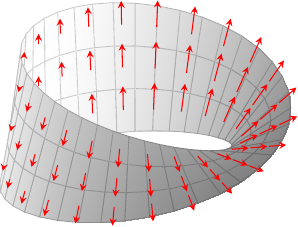}
        \caption{Vector field $q_*(Z)$.}
        \label{fig:mobiusz}
    \end{subfigure}
    \caption{Smooth vector fields defined on the Möbius strip $S$.}
    \label{fig:mobius}
\end{center}
\end{figure}

To see that they are globally defined, it is necessary to check the behavior of $q_*(X)$, $q_*(Y)$, $q_*(Z)$ at the points $[(0, v)]$. As $\cos(0) = -\cos(\pi) = 1$ and $\sin(0) = -\sin(\pi) = 0$, then
\[
    q_*(X_{(0, v)}) = q_*(X_{(1, -v)}),\quad q_*(Y_{(0, v)}) = q_*(Y_{(1, -v)}),\quad q_*(Z_{(0, v)}) = q_*(Z_{(1, -v)}).
\]
Therefore, $q_*(X), q_*(Y), q_*(Z)$ are well-defined.
    
We also need to check that they are smooth. Using the coordinate chart $(U, x)$, the fields $q_*(X)$, $q_*(Y)$, $q_*(Z)$ are locally described as
\[
    q_*(X) = \frac{\partial}{\partial x^1}, \quad q_*(Y) = \cos(\pi x^1)\frac{\partial}{\partial x^2}, \quad q_*(Z) = \sin(\pi x^1)\frac{\partial}{\partial x^2}.
\]
The description in the coordinate chart $(V, y)$ is analogous:
\[
    q_*(X) = \frac{\partial}{\partial y^1}, \quad q_*(Y) = \cos(\pi y^1)\frac{\partial}{\partial y^2}, \quad q_*(Z) = \sin(\pi y^1)\frac{\partial}{\partial y^2}.
\]
As $\cos(\pi(x - 1)) = -\cos(\pi x)$ and $\sin(\pi(x - 1)) = -\sin(\pi x)$, these coordinate representations are compatible with each other. Therefore, $q_*(X)$, $q_*(Y)$, $q_*(Z)$ are smooth vector fields.

To simplify notation, from now on the vector fields $q_*(X)$, $q_*(Y)$, $q_*(Z)$ will be denoted as $X$, $Y$, $Z$ respectively. These three vector fields are represented in Figure \ref{fig:mobius}. Then, using any auxiliary Riemannian metric $g$ on the Möbius strip, we obtain the following sections $w^1,w^2,w^3,w^4\in \Gamma(\mathbb TS)$ of the generalized tangent bundle:
\begin{equation}
\label{eq:trivializationmobiusstrip}
    w^1 = X, \quad w^2 = \flat_g X, \quad w^3 = Y - \flat_g Z, \quad w^4 = Z + \flat_g Y.  
\end{equation}

It must be checked that these sections are linearly independent at every point in $S$. It is clear that it suffices to check the linear independence of $w^3$ and $w^4$. At any point $[(0, v)]$, using the coordinate chart $(V,y)$ we have that $Y_{[(0, v)]} = \left.\frac{\partial}{\partial y^2}\right|_{[(0, v)]}$ and $Z_{[(0, v)]} = 0$, hence the linear independence of $w^3_{[(0, v)]}$ and $w^4_{[(0, v)]}$ ($w^3_{[(0, v)]}$ is in $TS$, whereas $w^4_{[(0, v)]}$ is in $T^*S$). At any $[(1/2, v)]$, with the coordinate chart $(U,x)$ it is immediate to see that $Y_{[(1/2, v)]} = 0$ and $Z_{[(1/2, v)]} = \left.\frac{\partial}{\partial x^2}\right|_{[(1/2, v)]}$, therefore $w^3_{[(1/2, v)]}$ is in $T^*S$ and $w^4_{[(1/2, v)]}$ is in $TS$. Finally, for $[(u, v)]$ with $u\neq 0, 1/2$, using the coordinate chart $(U,x)$ it is
\[
    \begin{split}
        w^3_{[(u, v)]} &= \cos(\pi u)\left.\frac{\partial}{\partial x^2}\right|_{[(u, v)]} - \sin(\pi u)\flat_g\left.\frac{\partial}{\partial x^2}\right|_{[(u, v)]}, \\
        w^4_{[(u, v)]} &= \sin(\pi u)\left.\frac{\partial}{\partial x^2}\right|_{[(u, v)]} + \cos(\pi u)\flat_g\left.\frac{\partial}{\partial x^2}\right|_{[(u, v)]}.
    \end{split}
\]

If $0 < u < 1/2$, then $\cos(\pi u) > 0$ and $\sin(\pi u) > 0$. Thus, there will not exist any nonzero $\lambda\in\R$ such that $\lambda w^3_{[(u, v)]} = w^4_{[(u, v)]}$. Analogously, if $1/2 < u < 1$, then $\cos(\pi u) < 0$ and $\sin(\pi u) > 0$, hence the linear independence of $w^3_{[(u, v)]}$ and $w^4_{[(x, y)]}$.

We have shown that $w^1$, $w^2$, $w^3$, $w^4$ are four sections of the generalized tangent bundle of the Möbius strip that are linearly independent. Therefore, $\mathbb TS$ is a trivial vector bundle.
\end{proof}

We can also find an example of a compact manifold whose tangent bundle is not trivial but with trivial generalized tangent bundle. This manifold is the Klein bottle, defined as the quotient
\[
    K = \R\times [-1, 1] / {\sim}, \quad (u,-1) \sim (u,1), \quad (u, v) \sim (u + m, (-1)^m v), \quad m\in \Z.
\]
Following an analogous procedure to the previous one, it can be seen that the Klein bottle $K$ admits four sections in $\Gamma(\mathbb TK)$ which are linearly independent at every point of $K$.

We study now the behavior of the generalized tangent bundle of the spheres $\Sph^n$. It is widely known that, even though $T\Sph^n$ is trivial only for $n = 1,3,7$, it is always \emph{stably trivial}: the Whitney sum of $T\Sph^n$ with the trivial bundle $\Sph^n\times \R$ is isomorphic to the trivial bundle $\Sph^n\times \R^{n+1}$. The following result, due to J. Allard \cite{ALLARD1980}, shows that the Whitney sum of several copies of a stably trivial vector bundle may be trivial.

\begin{proposition}[{\cite[Theorem 1.1]{ALLARD1980}}]
\label{prop:stablytrivialtheorem}
Let $E\to M$ be a stably trivial vector bundle such that $E\oplus (M\times \R^k)$ is isomorphic as a vector bundle to $M\times \R^m$. Then, $E\oplus\overset{r}{\ldots}\oplus E\to M$ is trivial for $r\geq k + \frac{k}{m-k}$.
\end{proposition}

As a consequence, the generalized tangent bundle of every sphere can be shown to be trivial.

\begin{corollary}
\label{cor:generalizedtangentbundlesn}
The generalized tangent bundle of the $n$-dimensional sphere $\Sph^n$ is trivial.
\end{corollary}

\begin{proof}
Since the vector bundles $T\Sph^n\oplus (\Sph^n\times \R)$ and $\Sph^n\times \R^{n+1}$ are isomorphic, as well as $\mathbb T\Sph^n$ and $T\Sph^n\oplus T\Sph^n$, then for any dimension $n$ we have
\[
	1 + \frac{1}{(n+1) - 1} = 1 + \frac{1}{n} \leq 2.
\]
Thus, the result is immediatly inferred from Proposition \ref{prop:stablytrivialtheorem}.
\end{proof}

The cases $n = 1,3,7$ were already analyzed in Subsection \ref{SUBSECTION:PARALLELIZABLEMANIFOLDS}: since $\Sph^1$, $\Sph^3$ and $\Sph^7$ are parallelizable, Proposition \ref{prop:parallelizableimpliestriviality} implies that their generalized tangent bundles are trivial.

\subsection{Non-parallelizable manifolds with non-trivial generalized tangent bundle}
\label{SUBSECTION:NONPARALLELIZABLEMANIFOLDSNONTRIVIALGENERALIZED}

Finaly, we study the existence of non-parallelizable smooth manifolds whose generalized tangent bundle is non-trivial. To this end, we make use of Stiefel–Whitney classes $w_k(\TM)\in H^k(M;\Z_2)$. An exhaustive description of these cohomology classes and their properties can be found in \cite{HATCHER2017, MILNORSTASHEFF1974}. These classes are invariant under vector bundle isomorphisms; hence, a necessary condition for a vector bundle $E\to M$ to be trivial is that $w_k(E) = 0$ for $k>0$.

It is important to recall the relation satisfied by the Stiefel-Whitney classes of the sum of two vector bundles: if $E\to M$ and $F\to M$ are vector bundles over $M$, then the Stiefel-Whitney classes of the Whitney sum $E\oplus F\to M$ satisfy, for every $k$,
\begin{equation}
	w_k(E\oplus F) = \sum_{i=0}^kw_i(E)w_{k-i}(F).
\label{eq:whitneysumclass}
\end{equation}
If we take the \emph{total Stiefel-Whitney class} $w(E) := 1 + w_1(E) + w_2(e) + \ldots$ of $E\to M$, then Eq. \eqref{eq:whitneysumclass} transforms into
\begin{equation}
	w(E\oplus F) = w(E)w(F).
\label{eq:whitneysumtotalclass}
\end{equation}

To find a non-parallelizable manifold whose generalized tangent bundle is non-trivial, we focus on real projective spaces, whose cohomology groups modulo 2 are widely known.

\begin{lemma}[{\cite[Lemma 4.3]{MILNORSTASHEFF1974}}]
The cohomology groups $H^k(\R\Prj^n;\Z_2)$ are cyclic of order 2 for $0\leq k\leq n$. Furthermore, each non-trivial group $H^k(\R\Prj^n;\Z_2)$ is generated by $a^k$, where $a$ is the non-zero element of $H^1(\R\Prj^n;\Z_2)$.
\end{lemma}

\begin{proposition}[{\cite[Theorem 4.5]{MILNORSTASHEFF1974}}]
The total Stiefel-Whitney class of the tangent bundle of $\R\Prj^n$ is
\begin{equation}
	w(T(\R\Prj^n)) = (1 + a)^{n+1} = 1 + \binom{n+1}{1}a + \binom{n+1}{2}a^2 + \ldots + \binom{n+1}{n} a^n,
	\label{eq:stiefelwhitneyrpn}
\end{equation}
where $a$ is the non-zero element of $H^1(\R\Prj^n; \Z_2)$ and the binomial coefficients are taken modulo $2$.
\end{proposition}

These results allow us to prove that $\mathbb T\R\Prj^n$ is non-trivial when $n+1$ is not a power of $2$.

\begin{proposition}
The generalized tangent bundle of the real projective space $\R\Prj^n$ is non-trivial if $n+1$ is not a power of $2$.
\end{proposition}

\begin{proof}
If $n+1$ is not a power of $2$, then we can write $n + 1 = 2^km$ for a certain integer $k\geq 0$ and an odd number $m>1$. Using Eq. \eqref{eq:stiefelwhitneyrpn}, we compute the total Stiefel-Whitney class of the tangent bundle of $\R\Prj^n$:
\[
	w(T\R\Prj^n) = (1 + a)^{n+1} = (1 + a)^{2^km} = (1 + a^{2^k})^m.
\]
Since the tangent and cotangent bundle of $\R\Prj^n$ are diffeomorphic, their total Stiefel-Whitney classes are equal, that is, $w(T\R\Prj^n) = w(T^*\R\Prj^n)$. Therefore, using Eq. \eqref{eq:whitneysumtotalclass} we obtain
\[
	\begin{split}
		w(\mathbb T\R\Prj^n) &= w(T\R\Prj^n) w(T^*\R\Prj^n) = w(T\R\Prj^n)^2 = (1 + a^{2^k})^{2m} = (1 + a^{2^{k+1}})^m  \\
							 & = 1 + \binom{m}{1}a^{2^{k+1}} + \binom{m}{2}a^{2\cdot 2^{k+1}} + \ldots = 1 + m a^{2^{k+1}} + \frac{m(m-1)}{2}a^{2\cdot 2^{k+1}} + \ldots
	\end{split}
\]
Since $m$ is odd, $m\equiv 1\ \mathrm{mod\ 2}$. On the other hand, $m > 1$ and hence $2^{k+1} = 2^k\cdot 2 < 2^km$. Therefore, $\binom{m}{1}a^{2^{k+1}}\neq 0$ and thus the total class $w(\mathbb T\R\Prj^n)$ is not trivial.
\end{proof}

It is worth remarking the differences between the tangent bundle and the generalized tangent bundle of spheres and projective spaces, as can be seen in Table \ref{tab:introduction}. For $n=1,3,7$, both $\Sph^n$ and $\R\Prj^n$ are parallelizable manifolds, and therefore their generalized tangent bundles are trivial by Proposition \ref{prop:parallelizableimpliestriviality}. For $n\neq 2^k-1$, neither $\Sph^n$ nor $\R\Prj^n$ is parallelizable; nevertheless, $\mathbb T\Sph^n$ is trivial, whereas $\mathbb T\R\Prj^n$ is not.

\section{Relation of triviality with generalized geometric structures}
\label{SECTION:RELATIONOFTRIVIALITYWITHGENERALIZEDGEOMETRICSTRUCTURES}

As is well known, the triviality of the tangent bundle of a manifold entails significant topological consequences. In particular, the manifold must be orientable, and if its dimension is even, it admits an almost complex structure. Our aim is to establish analogous results concerning the triviality of the generalized tangent bundle.

Proposition \ref{prop:trivialitymobius} shows that the triviality of $\TM$ does not imply that $M$ is orientable, which constitutues an important difference with respect to $TM$. However, geometric structures can still be defined on $\TM$. For example, using a trivialization $w^1,\ldots,w^{2n}\in \GTM$ of $\TM$ one can obtain a \emph{weak generalized almost complex structure}, i.e., a vector bundle endomorphism $\J\colon \TM\to \TM$ such that $\J^2 = -\mathcal Id$. One such endomorphism can be defined by setting, for each $p\in M$
\begin{equation}
\label{eq:inducedgenalmostcomplexstructure}
	\J(w^{2i-1}_p) = w^{2i}_p,\quad \J(w^{2i}_p) = -w^{2i-1}_p,\quad i = 1,2,\ldots,n.
\end{equation}
This is not the only geometric structure that can be obtained: if we define the endomorphism $\F\colon \TM\to \TM$ as
\begin{equation}
\label{eq:inducedgenalmostparacomplexstructure}
	\F(w^{2i-1}_p) = w^{2i}_p,\quad \F(w^{2i}_p) = w^{2i-1}_p,\quad i = 1,2,\ldots,n,
\end{equation}
we obtain a \emph{weak generalized almost paracomplex structure}, i.e., a vector bundle endomorphism such that $\F^2 = \mathcal Id$ and whose $+1$-eigenbundle has rank $n$.

In general, it is challenging to obtain explicit expressions for the $2n$ sections $w^1,\ldots,w^{2n}$. Nevertheless, Propositions \ref{prop:parallelizableimpliestriviality} and \ref{prop:trivialitymobius} describe these sections for parallelizable manifolds and for the Möbius strip, respectively. This makes it possible to particularize Eqs. (\ref{eq:inducedgenalmostcomplexstructure}, \ref{eq:inducedgenalmostparacomplexstructure}) to these cases.

First, suppose that $M$ is parallelizable and $g$ is a metric on $M$. Then, Proposition \ref{prop:parallelizableimpliestriviality} provides $2n$ sections $X^1,\flat_g X^1,\ldots,X^n,\flat_g X^n$, where the $n$ vector fields $X^1,\ldots,X^n\in \xm$ are pointwise linearly independent and $\flat_g\colon TM\to T^*M$ is given by $(\flat_gX)Y = g(X,Y)$. In this case, the weak generalized almost complex structure $\J\colon \TM\to \TM$ from Eq. \eqref{eq:inducedgenalmostcomplexstructure} is
\[
	\J X^i_p = \flat_g X^i_p,\quad \J(\flat_g X^i_p) = -X^i_p,\quad i = 1,2,\ldots,n.
\]
Therefore, if we write $\sharp_g = \flat_g^{-1}$ then it can be immediately seen that $\J$ takes the form
\begin{equation}
\label{eq:jg}
	\J(X + \xi) = - \sharp_g\xi + \flat_g X,
\end{equation}
for each $X\in T_pM$ and $\xi\in T_p^*M$. Similarly, the weak generalized almost paracomplex structure $\F\colon \TM\to \TM$ from Eq. \eqref{eq:inducedgenalmostparacomplexstructure} is
\begin{equation}
\label{eq:fg}
	\F(X + \xi) = \sharp_g\xi + \flat_g X.
\end{equation}

In the case of the Möbius strip $S$, by Proposition \ref{prop:trivialitymobius} we have the trivialization $w^1 = X$, $w^2 = \flat_g X$, $w^3 = Y - \flat_g Z$, $w^4 = Z + \flat_gY$, given in Eq. \eqref{eq:trivializationmobiusstrip}. Then, the endomorphism $\J\colon \mathbb TS\to \mathbb TS$ from Eq. \eqref{eq:inducedgenalmostcomplexstructure} is
\[
	\J X_p = \flat_g X_p,\quad \J(\flat_gX_p) = -X_p,\quad \J(Y_p - \flat_g Z_p) = Z_p + \flat_g Y_p,\quad \J(Z_p + \flat_g Y_p) = -Y_p + \flat_g Z_p.
\]
From these identities, it follows that $\J(X + \xi) = -\sharp_g\xi + \flat_g X$, which has the same expression as Eq. \eqref{eq:jg}. Similarly, the weak generalized almost paracomplex structure $\F\colon \TM\to \TM$ from Eq. \eqref{eq:inducedgenalmostparacomplexstructure} is $\F(X + \xi) = \sharp_g\xi + \flat_g X$, as in Eq. \eqref{eq:fg}.

The structure $\J$ from Eq. \eqref{eq:jg} has been studied by A. Nannicini \cite{NANNICINI2010}, whereas $\F$ from Eq. \eqref{eq:fg} has been studied by C. Ida and A. Manea \cite{IDAMANEA2017}. Both $\J$ and $\F$ are defined on the generalized tangent bundle of every (pseudo-)Riemannian manifold $(M, g)$, without requiring $\TM$ to be trivial, and their properties are remarkable in their own right. However, here we show that these structures also arise naturally in the present context of the trivialization of $\TM$ when $M$ is parallelizable, and that of the proof of Proposition \ref{prop:trivialitymobius} concerning the Möbius strip.

Generalized almost complex and generalized almost paracomplex structures were at first required to be compatible with the \emph{canonical pairing} $\G_0\in \Gamma((\TM)^*\otimes(\TM)^*)$ \cite{GUALTIERI2011, WADE2004}, defined as
\[
	\G_0(X + \xi, Y + \eta) = \frac{1}{2}(\xi(Y) + \eta(X)),
\]
for each $X,Y\in T_pM$ and $\xi,\eta\in T^*_pM$. Given a bundle endomorphism $\mathcal K\colon \TM\to \TM$, the compatibility condition is
\begin{equation}
	\G_0(\mathcal K(X + \xi), Y + \eta) = -\G_0(X + \xi, \mathcal K(Y + \eta)).
\end{equation}
Those generalized geometric structures satisfying this skew-symmetry condition were called by the present authors \emph{strong structures}, whereas those that do not are referred to as \emph{weak structures} (see \cite{ETAYOGOMEZNICOLASSANTAMARIA2025,GOMEZNICOLAS2025}). The following result shows that, even though the structures $\J$ and $\F$ from Eqs. (\ref{eq:jg}, \ref{eq:fg}) are not strong structures, they are closely related to the canonical pairing.

\begin{proposition}[{\cite[Proposition 4.5]{ETAYOGOMEZNICOLASSANTAMARIA2024}}]
The weak generalized almost complex structure $\J$ from Eq. \eqref{eq:jg} and the weak generalized almost paracomplex structure $\F$ from Eq. \eqref{eq:fg} are symmetric with respect to $\G_0$, that is to say,
\[
	\G_0(\J(X + \xi), Y + \eta) = \G_0(X + \xi, \J(Y + \eta)),\quad \G_0(\F(X + \xi), Y + \eta) = \G_0(X + \xi, \F(Y + \eta)),
\]
for every $X, Y\in T_pM$ and $\xi,\eta\in T_p^*M$.
\end{proposition}

\section*{Acknowledgements}
Part of this note is based on material from the doctoral thesis of one of the authors \cite{GOMEZNICOLAS2025}. This author, P. G.-N., would like to express his sincere gratitude to his supervisors, Fernando Etayo and Rafael Santamaría, for their invaluable guidance and encouragement throughout the development of this research. The authors would also like to thank Roberto Rubio for suggesting several questions that were worthy of further study.


\end{document}